\documentclass[letterpaper, 10pt, conference]{IEEEtran}
\IEEEoverridecommandlockouts
\usepackage{scrextend}
\usepackage{amsmath,amssymb,amsfonts}
\usepackage{mathtools}
\usepackage{graphicx}
\usepackage[inkscapelatex=false]{svg}
\usepackage{orcidlink}

\usepackage{bm}
\usepackage{enumitem}
\usepackage{cite}
\usepackage{cleveref}
\usepackage[dvipsnames]{xcolor}
\usepackage{comment}
\usepackage{amsthm}

\newtheorem{theorem}{Theorem}

\newtheorem{definition}{Definition}

\newtheorem{remark}{Remark}
\newtheorem{assumption}{Assumption}
\newtheorem{lemma}{Lemma}
\newtheorem{proposition}{Proposition}
\newtheorem{corollary}{Corollary}

\Crefname{property}{Property}{Properties}
\Crefname{assumption}{Assumption}{Assumptions}

\newcommand{\infconv}{\mathbin{\square}}

\allowdisplaybreaks

\usepackage{tikz}

\begin{document}
\title{Target Mirror Descent: A Unifying Framework for Solving Monotone Variational Inequalities}
\author{Yu-Wen Chen, 
Can Kizilkale,
Murat Arcak
\thanks{Yu-Wen Chen, Can Kizilkale, and Murat Arcak are with the Department of Electrical Engineering and Computer Sciences, University of California, Berkeley, CA, USA. {\tt\small \{yuwen\_chen, cankizilkale, arcak\}@berkeley.edu}.}%
}

\maketitle
\begin{abstract}
It is well known that mirror descent may diverge or cycle on merely monotone variational inequalities.
In this paper, we propose \emph{Target Mirror Descent} (TMD), a unified framework that stabilizes monotone flows via a target point correction mechanism in the dual update. By appropriate design choices, TMD recovers the proximal point algorithm, extragradient methods, splitting methods, Brown-von Neumann-Nash dynamics, forward-backward-forward dynamics, and discounted mirror descent as special cases.
Thus, we establish a unified perspective on these landmark algorithms and their convergence. Beyond unification, we leverage the TMD framework to correct an equilibrium misalignment in discounted mirror descent and to generalize its higher-order extension beyond interior solutions. Moreover, a key structural feature of TMD is the explicit decoupling of the mirror map from the target determination, which enables \emph{geometric ensembles}: multiple algorithms solve the same problem in parallel using distinct mirror maps, while sharing a common dual update. We show that such an ensemble rigorously reduces to a single TMD with a synthesized mirror map, and thus inherits these convergence guarantees.
\end{abstract}


\section{Introduction}
\label{sec:intro}

Variational inequalities (VIs) provide a powerful mathematical framework for modeling equilibrium 
problems in constrained systems, encompassing optimization, game theory, and physical systems 
as special cases. Unlike scalar minimization, VIs are governed by a vector field $F$ that may lack a potential function, arising naturally in settings such as multi-agent games, traffic equilibria, and resource allocation.

A central challenge is designing algorithms that converge reliably when $F$ is merely monotone --- it is known that gradient (or mirror) descent may cycle or diverge in this regime \cite{Korpelevich1976}. To address this, mirror descent can be shown to converge in the ergodic (time-averaged) sense~\cite{Mertikopoulos2018a}, but last-iterate convergence, which is of greater practical interest, requires more sophisticated algorithms. To this end, several stabilizing algorithms have been developed independently, each with its own convergence analysis: the proximal point algorithm \cite{Eckstein1993,ELFAROUQ2001}, extragradient methods \cite{Korpelevich1976,Nemirovski2004,Dang2015,Diakonikolas2021}, splitting methods \cite{Facchinei2003,Bauschke2017,Bui2021}, Brown-von Neumann-Nash (BNN) dynamics \cite{Sandholm2010}, forward-backward-forward dynamics \cite{Bot2020,Tseng2000}, and (improved) discounted mirror descent \cite{Gao2024a}. While \cite{Mokhtari2020} unifies extragradient and optimistic gradient methods as approximations of the proximal point method, a broader unifying perspective has remained elusive, particularly regarding the connections between game-theoretic dynamics such as BNN and classical optimization algorithms. Moreover, all of these algorithms operate with a single fixed geometric structure, leaving open the question of whether multiple geometries can be combined to yield more robust behavior.

In machine learning, ensemble methods improve robustness by combining multiple models. The underlying intuition, often called the \emph{wisdom of crowds}, is that a diverse collection of models captures complementary specialization that no single model can provide alone. In the context of mirror descent, this motivates a \emph{geometric ensemble}: multiple algorithms, each equipped with a geometrically distinct mirror map, solving the same VI problem in parallel.

The main contributions of this paper are as follows:
\begin{itemize}
    \item We propose \emph{Target Mirror Descent} (TMD), a unified framework that stabilizes 
    monotone flows via a target point correction mechanism in the dual update, with a rigorous convergence analysis.
    \item We demonstrate that, by appropriate design choices, TMD subsumes proximal point algorithms, extragradient methods, splitting methods, BNN dynamics, forward-backward-forward dynamics, discounted mirror descent and its higher-order extension as special cases, establishing a unified structure and analysis across all of them.
    \item A key structural feature of TMD is the explicit decoupling of the mirror map from the 
    target determination. This separation enables \emph{geometric ensembles}: multiple TMD algorithms with distinct mirror maps, whose dual updates are driven by a shared ensemble state. We show that the ensemble rigorously reduces to a single TMD with a synthesized mirror map, inheriting convergence guarantees.
\end{itemize}

The remainder of this paper is organized as follows. \Cref{sec:pre} introduces preliminaries on mirror descent and variational inequalities. \Cref{sec:prob} presents the TMD framework, its convergence analysis, and the unifying connections to landmark algorithms. Section~\ref{sec:GE} develops the geometric ensemble and presents numerical validation.

\medskip
\noindent\textit{Notation:} 
$\langle\cdot,\cdot\rangle$ denotes the inner product, $\top$ the transpose, $\nabla f$ and $\partial f$ the gradient and subdifferential of $f$, $\circ$ the composition of functions, $\mathcal{N}_\mathcal{X}(x)$ the normal cone of $\mathcal{X}$ at $x$, $\operatorname{Im}(f)$ the image of $f$, $\operatorname{int}(\mathcal{X})$ the interior of $\mathcal{X}$, $\Delta^n$ the $n$-dimensional probability simplex, and $\mathbb{R}^n_{>0}$ the positive orthant.

\section{Preliminaries}
\label{sec:pre}
\subsection{Mirror Descent}
Mirror descent generalizes gradient descent by operating in a dual space induced by a strictly convex function $h:\mathcal{X}\to\mathbb{R}$, where $\mathcal{X}\subseteq\mathbb{R}^n$ is nonempty, closed, and convex. Denoting the convex conjugate of $h$ by $h^*$, the map $\nabla h:\mathcal{X}\to\mathbb{R}^n$ sends primal iterates to the dual space, where gradient updates are performed; $\nabla h^*:\mathbb{R}^n\to\mathcal{X}$ then maps back to the primal space, referred to as the \emph{mirror map} in this paper. The Bregman divergence, $D_h(x, y) = h(x) - h(y) - \langle \nabla h(y), x - y \rangle$, quantifies how much a function curves away from its linear approximation. When paired with a carefully chosen mirror map, mirror descent replaces standard Euclidean distance with this divergence, naturally considering the geometric landscape.

\subsection{Variational Inequalities}
\begin{definition}[Variational Inequality]\label{def:VI}
    Given a set $\mathcal{X}\subseteq\mathbb{R}^n$ and a map $F:\mathcal{X}\to\mathbb{R}^n$, an $x^*\in \mathcal{X}$ solves the variational inequality, denoted by $\operatorname{VI}(\mathcal{X},F)$, if $F(x^*)^\top(x-x^*) \geq 0 , \forall x \in \mathcal{X}$.
    The set of all solutions is denoted as $\operatorname{SOL}(\mathcal{X},F)$.
\end{definition}


Some common properties studied for $F$ are as follows.

\begin{definition}[Monotonicity]\label{def:mono}
    Consider a map $F:\mathcal{X}\to\mathbb{R}^n$.\\
    (i) $F$ is pseudo-monotone if
    \begin{align}
        F(y)^\top(x-y)\geq0\ \Rightarrow\ F(x)^\top(x-y)\geq0,\ \forall x,y\in\mathcal{X}. \label{eq:pseudo-monotone}
    \end{align}
    (ii) $F$ is monotone if
    \begin{align}
        \left(F(x)-F(y)\right)^\top\!(x-y)\geq 0, \quad \forall x,y\in \mathcal{X}. \label{eq:monotone}
    \end{align}
    (iii) $F$ is $\sigma$-strongly monotone if, for $\sigma>0$ and a norm $\|\cdot\|$,
    \begin{align}
        \left(F(x)-F(y)\right)^\top\!(x-y)\geq \sigma\|x-y\|^2, \quad \forall x,y\in \mathcal{X}. \label{eq:strongly-monotone}
    \end{align}
\end{definition}


\begin{definition}[Variational Stability]\label{def:VS}
    An $\bar{x}\in\mathcal{X}$ is variationally stable (VS) w.r.t. $F:\mathcal{X}\to\mathbb{R}^n$, if $F(x)^\top\!(x-\bar{x})\geq0, \forall x\in\mathcal{X}$.
    A set $\mathcal{S}\subseteq\mathcal{X}$ is VS w.r.t. $F$, if the inequality holds $\forall\bar{x}\in \mathcal{S}$.
\end{definition}

Note that strong monotonicity implies monotonicity, which in turn implies pseudo-monotonicity.
Moreover,  if $F$ is pseudo-monotone, then $\operatorname{SOL}(\mathcal{X},F)$ is VS w.r.t. $F$.

\section{Target Mirror Descent: A Unified Framework}
\label{sec:prob}

In this paper, we consider variational inequalities with closed and convex $\mathcal{X}$ and continuous $F$. As noted in [1] and [8, Chapter~12.1], gradient (or mirror) descent may diverge or cycle when applied to merely monotone $F$, which motivates our proposed TMD framework.








In \Cref{subsec:TMD}, we introduce the \emph{Target Mirror Descent} (TMD) framework to address this instability and provide the convergence analysis in \Cref{subsec:convergence}. Beyond stability, we demonstrate in \Cref{subsec:connections} that TMD serves as a unifying framework which subsumes many landmark algorithms.
Throughout, we assume well-posedness:
\begin{assumption}
    $\operatorname{SOL}(\mathcal{X},F)\neq \emptyset$.
\end{assumption}

\subsection{Target Mirror Descent (TMD) dynamics}
\label{subsec:TMD}

We propose the \emph{target mirror descent (TMD)} dynamics:
\begin{subequations}\label{eq:TMD}
\begin{align}
    \dot{z}(t) & = \alpha \left(S(T(x(t)))-S(x(t))\right) - \beta \Phi(x(t))  \label{eq:dual_update} \\
    x(t) &= \nabla h^* (z(t)),\label{eq:dual_map}
\end{align}
\end{subequations}
where $\alpha>0$ and $\beta\geq0$ are parameters, $\nabla h^*:\mathbb{R}^n\to\mathcal{X}$ is the mirror map, and $T:\mathcal{X}\to\mathcal{X}$, $S: \mathcal{X}\to\mathbb{R}^n$, and $\Phi:\mathcal{X}\to\mathbb{R}^n$ are design functions satisfying:%
\smallskip
\begin{enumerate}[label={[C\arabic*]}, ref=C\arabic*, leftmargin=3em, align=left]
    \item \label{cond:S} $S$ is $\sigma$-strongly monotone w.r.t. some norm $\|\cdot\|$.
    \item \label{cond:VS} There exists an $\bar{x}\in\mathcal{X}$ which is VS w.r.t. $\Phi$.
    \item \label{cond:surrogate} $\operatorname{SOL}(\mathcal{X},\Phi)\subseteq\operatorname{SOL}(\mathcal{X},F)$.
    \item \label{cond:ST} $T=(S+\Phi+\mathcal{N}_{\mathcal{X}})^{-1}\circ S$ is well-defined,
\end{enumerate}
\smallskip%
where $\mathcal{N}_\mathcal{X}(x)$ is the normal cone of $\mathcal{X}$ at $x$. By embedding $\mathcal{N}_\mathcal{X}$ directly into the formulation of $T$, the system implicitly enforces feasibility, acting as a projection that ensures $\operatorname{Im}(T)\subseteq \mathcal{X}$.
We provide a sufficient condition for \ref{cond:ST}; weaker conditions can be found in \cite{Harker1990}.

\begin{lemma}[{\cite[Corollary~3.2]{Harker1990}}] \label{lem:well-defined}
    If $S+\Phi$ is strongly monotone, then $T$ is well-defined.
\end{lemma}

To clarify the mechanics of the TMD dynamics \eqref{eq:TMD}, consider the case where $\alpha=0$ and $\Phi=F$; the framework then reduces exactly to mirror descent. The term $S \circ T - S$ acts as a correction operator that enhances stability. Specifically, for a given state $x$, the mechanism computes a \emph{target point} $T(x)$. The operator $S$ then maps both $x$ and $T(x)$ to the dual space, and the dual discrepancy $S(T(x)) - S(x)$ provides a corrected update direction. The formulation admits a clear graphical interpretation; \Cref{fig:TMD} illustrates the idea.

While the selection of $S$, $\Phi$, and $T$ offers significant flexibility, they are constrained by functional requirements: $S$ maps between spaces, $\Phi$ serves as a surrogate function of $F$, and the pair $S$ and $\Phi$ must collectively define the target mechanism $T$, as per conditions \ref{cond:S}, \ref{cond:surrogate}, and \ref{cond:ST}. We thoroughly demonstrate the broad applicability of TMD in \Cref{subsec:connections}. Here, we merely illustrate a baseline scenario: an $L$-Lipschitz continuous and pseudo-monotone $F$. For any $\sigma$-strongly monotone $S$ and $\Phi = \eta F$ with $0<\eta < \frac{\sigma}{L}$, $S+\Phi$ is strongly monotone, rendering $T$ well-defined by \Cref{lem:well-defined}. Thus, all \ref{cond:S}--\ref{cond:ST} are satisfied.

\begin{figure}[h!]
    \centering
    \includegraphics[width=.8\columnwidth]{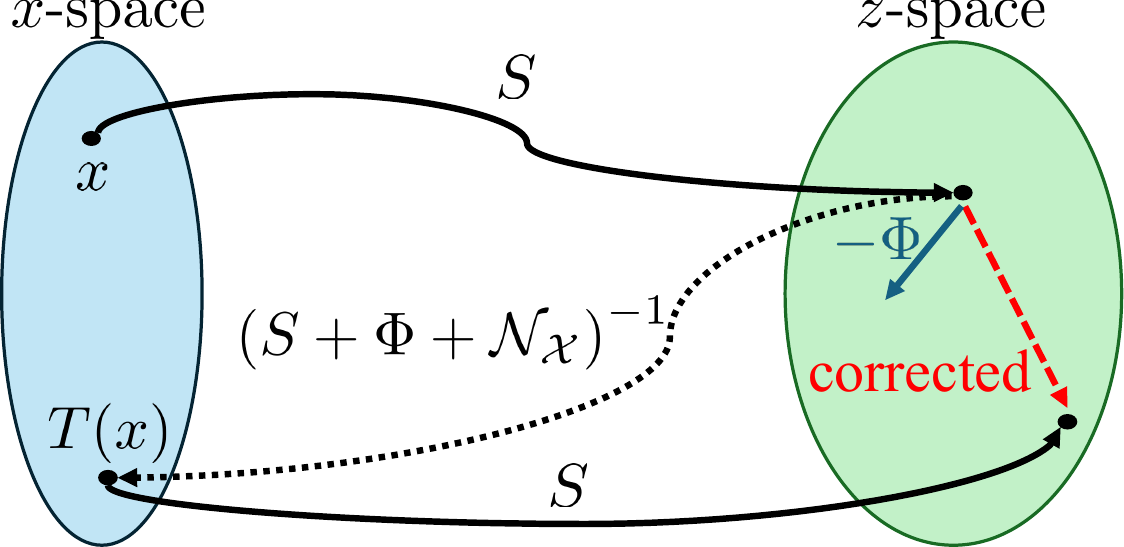}
    \caption{Graphical illustration of the TMD dynamics \eqref{eq:TMD}.
    The target point $T(x)$ is computed in the primal space and mapped to the dual via $S$, producing the corrected update direction $S(T(x))-S(x)$, marked by the red dashed line.}
    \label{fig:TMD}
\end{figure}

A key feature of TMD dynamics \eqref{eq:TMD} is the decoupling of the mirror map $\nabla h$ from $S,T,\Phi$ in dual update \eqref{eq:dual_update}. Essentially, this retains the structure of mirror descent, except that $F$ is preconditioned by the target correction \eqref{eq:dual_update}, enhancing the stability. This decoupling allows geometric ensembles, where multiple algorithms maintain distinct mirror maps while sharing a common dual update, as detailed in \Cref{sec:GE}. Many existing algorithms stabilizing monotone flows preclude this benefit. For example, in \Cref{subsec:connections}, we demonstrate that while extragradient (mirror-prox) methods belong to the TMD framework, their choices for $S$ and $T$ are bound to $\nabla h$, which excludes the possibility of geometric ensembles.

\subsection{Convergence analysis for TMD}
\label{subsec:convergence}

\begin{theorem} \label{thm:asymp}
    TMD asymptotically approaches $\operatorname{SOL}(\mathcal{X},\Phi)$.
\end{theorem}

\begin{proof}
    Consider the Lyapunov function $V(x)=D_h(\bar{x},x)$, where $D_h$ denotes the Bregman divergence.
    \begin{subequations}\label{eq:Vdot}
    \begin{align}
        \dot{V}&(x(t))=-\left\langle\nabla^2 h (x(t))\dot{x}(t), \bar{x}-x(t)\right\rangle \\
        &=-\left\langle \dot{z}(t), \bar{x}-x(t)\right\rangle \label{eq:apply_dual}\\
        &=-\alpha\left\langle (S\circ T-S)(x(t)), \bar{x}-x(t)\right\rangle\nonumber\\
        &\hspace{3em}-\beta\langle \Phi(x(t)), x(t)-\bar{x}\rangle \label{eq:apply_zdot}\\
        &\leq -\alpha\left\langle (S\circ T-S)(x(t)), \bar{x}-T(x(t))\right\rangle\nonumber\\
        &\hspace{3em}-\alpha\left\langle (S\circ T-S)(x(t)), T(x(t))-x(t)\right\rangle \label{eq:apply_VS}\\
        &=-\alpha\left\langle\Phi(T(x(t))), T(x(t))-\bar{x}\right\rangle+\alpha\langle \theta(t),\bar{x}-T(x(t))\rangle\nonumber\\
        &\hspace{3em}-\alpha\left\langle S(T(x(t)))-S(x(t)), T(x(t))-x(t)\right\rangle \label{eq:apply_ST}\\
        &\leq -\alpha\sigma \|T(x(t))-x(t)\|^2 \label{eq:apply_VS_Monotone},
    \end{align}
    \end{subequations}
    where $\theta(t)\in\mathcal{N}_\mathcal{X}(T(x(t)))$. Note that \eqref{eq:apply_dual} follows from \eqref{eq:dual_map}, \eqref{eq:apply_VS} from \ref{cond:VS}, \eqref{eq:apply_ST} from \ref{cond:ST}, and \eqref{eq:apply_VS_Monotone} from the property of normal cone, \ref{cond:S}, and \ref{cond:VS}.
    By Barbalat's lemma, $x(t)$ asymptotically approaches the set of fixed points of $T$.

    It remains to show that the fixed points of $T$ are contained in $\operatorname{SOL}(\mathcal{X},\Phi)$. From \ref{cond:ST}, with $y=Tx$, we have
    \begin{align}
        S(x)\in S(y)+\Phi(y)+\mathcal{N}_\mathcal{X}(y).
    \end{align}
    By the definition of the normal cone, we obtain
    \begin{align}
        \langle S(x)-S(y)-\Phi(y), u-y\rangle\leq0, \quad \forall u\in\mathcal{X}.\label{eq:fixed_point}
    \end{align}
    For any fixed point $x^*$ of $T$, substituting $x=x^*$ and $y=Tx^*=x^*$ into \eqref{eq:fixed_point} yields $\langle\Phi(x^*), u-x^*\rangle\geq0, \forall u\in\mathcal{X}$. Therefore, the theorem follows.
\end{proof}


Asymptotic approach to a solution set does not guarantee convergence to a specific point within it. To bridge this gap, we consider the stronger condition given as follows.
\smallskip
\begin{enumerate}[label={[C\arabic*$^+$]}, ref=C\arabic*$^+$, leftmargin=3em, start=2, align=left]
    \item \label{cond:VS+} $\operatorname{SOL}(\mathcal{X},\Phi)$ is VS w.r.t. $\Phi$.
\end{enumerate}
\smallskip%
One example satisfying \ref{cond:VS+} is the pseudo-monotone $F$, which follows from the discussion after \Cref{def:VS}.

\begin{corollary}\label{cor:converge}
    TMD with \ref{cond:VS+} converges to $x^*\in \operatorname{SOL}(\mathcal{X},\Phi)$.
\end{corollary}

\begin{proof}
    By \ref{cond:VS+}, the Lyapunov function $V(x)=D_h(\bar{x},x)$ is valid for \emph{any} $\bar{x}\in \operatorname{SOL}(\mathcal{X},\Phi)$. In other words, $x(t)$ is Bregman-Fej\'er monotone w.r.t. $\operatorname{SOL}(\mathcal{X},\Phi)$ \cite{Bauschke2003}. Thus, by \cite[Theorem~4.11]{Bauschke2003}, $x(t)$ converges to a point in the set.
\end{proof}


The proof of \Cref{thm:asymp} reveals that the term $\|T(x)-x\|^2$ in \eqref{eq:apply_VS_Monotone} is strictly positive whenever $x\notin\operatorname{SOL}(\mathcal{X},\Phi)$. 
This positivity can compensate for a mild violation of \ref{cond:VS}, e.g., when no $\bar{x}$ is 
perfectly VS w.r.t.\ $\Phi$, leading to the following relaxed condition.
\smallskip
\begin{enumerate}[label={[C\arabic*$^-$]}, ref=C\arabic*$^-$, leftmargin=3em, start=2, align=left]
\item \label{cond:VS-} There exists an $\bar{x} \in \mathcal{X}$ s.t. $\forall x \notin\operatorname{SOL}(\mathcal{X},\Phi)$,
\begin{align}
    &\alpha\left(\sigma\|T(x)-x\|^2 + \langle \Phi(T(x)), T(x)-\bar{x} \rangle\right) \nonumber\\
    &\hspace{9em} + \beta\langle \Phi(x), x-\bar{x} \rangle > 0. \label{eq:relax}
\end{align}
\end{enumerate}
Note that \ref{cond:VS} implies \ref{cond:VS-} trivially.

\begin{corollary} \label{cor:VS-}
    TMD with \ref{cond:VS-} asymptotically approaches $\operatorname{SOL}(\mathcal{X},\Phi)$. If \ref{cond:VS-} holds for all $\bar{x}\in\operatorname{SOL}(\mathcal{X},\Phi)$, then TMD converges to $x^*\in\operatorname{SOL}(\mathcal{X},\Phi)$.
\end{corollary}

Finally, we note that the TMD dynamics \eqref{eq:dual_update} admits a natural extension by incorporating a design input $u(t)$:
\begin{align}
    \dot{z}(t)=\alpha(S\circ T-S)(x(t))-\beta\Phi(x(t))+u(t). \label{eq:TMD_u}
\end{align}
It follows from \eqref{eq:Vdot} that \eqref{eq:TMD_u} is output strictly Equilibrium Independent Passive (EIP) \cite{Hines2011} from input $u$ to output $x$, with storage function $D_h(\bar{x},x)$. This allows \eqref{eq:TMD_u} to interconnect with any EIP system while preserving convergence. The use of \eqref{eq:TMD_u} is deferred to \Cref{rem:MD2}, after we first establish the broad applicability of TMD with $u\equiv 0$ in the following subsection.

\subsection{TMD as a unifying framework}
\label{subsec:connections}

A central contribution of TMD is its ability to subsume several landmark algorithms as special cases. To facilitate connections with these iterative algorithms, we provide the discrete-time representation of \eqref{eq:dual_update} in terms of $x$:\footnote{While discretization introduces a step size $\delta$, it can be absorbed w.l.o.g. by the coefficients, e.g., redefine $\alpha$ to be $\delta \alpha$. Thus, we omit $\delta$ in \eqref{eq:dual_discrete}.}
\begin{align}
    x_{t+1}=\nabla h^*\!\left(\nabla h(x_t) + \alpha (S\circ T-S)(x_t) - \beta \Phi(x_t)\right). \label{eq:dual_discrete}
\end{align}
Unless otherwise noted, we assume $F$ Lipschitz continuous and pseudo-monotone and $\nabla h$ strongly monotone. Let $\eta>0$.

\subsubsection{(Nonlinear) Proximal Point Algorithms (PPA)} \cite{Eckstein1993,ELFAROUQ2001}
\begin{align}
        x_{t+1} = \left(\nabla h + \eta F+\mathcal{N}_\mathcal{X}\right)^{-1}\!\left(\nabla h(x_t)\right). \label{eq:PPA}
\end{align}

\begin{proposition} \label{prop:PPA}
    Let $\alpha=1$, $\beta=0$, $S=\nabla h$, and $\Phi=\eta F$. Then, \eqref{eq:dual_discrete} recovers \eqref{eq:PPA}. Moreover, the convergence follows from \Cref{cor:converge}.
\end{proposition}

\begin{proof}
There exists a sufficiently small $\eta$ s.t. $\nabla h+\eta F$ is strongly monotone. Thus, $T=(\nabla h+\eta F+\mathcal{N}_\mathcal{X})^{-1}\circ \nabla h$ is well-defined. All \ref{cond:S}--\ref{cond:ST} and \ref{cond:VS+} are satisfied. Substituting the selections into \eqref{eq:dual_discrete} yields \eqref{eq:PPA}.
\end{proof}

\subsubsection{Extragradient Methods (EG) / Mirror-Prox} \cite{Korpelevich1976,Nemirovski2004,Dang2015}
\begin{subequations}\label{eq:EG}
\begin{align}
    \omega_t&= \arg\min_{y\in\mathcal{X}}\ \left\{\langle \eta F(x_t),y\rangle+D_h(y,x_t)\right\}\label{eq:EG_first} \\
    x_{t+1}&=\arg\min_{y\in\mathcal{X}}\ \left\{\langle \eta F(\omega_t), y\rangle+D_h(y,x_t)\right\},\label{eq:EG_second}
\end{align}
\end{subequations}

\begin{proposition} \label{prop:EG}
    Let $\alpha=1$, $\beta=0$, $S=\nabla h-\eta F$, and $\Phi=\eta F$. Then, \eqref{eq:dual_discrete} recovers \eqref{eq:EG}. Moreover, the convergence follows from \Cref{cor:converge}.
\end{proposition}

\begin{proof}
    There exists a sufficiently small $\eta$ s.t. $S=\nabla h-\eta F$ is strongly monotone and $T=(\nabla h+\mathcal{N}_\mathcal{X})^{-1}\circ(\nabla h-\eta F)=\nabla h^*\circ(\nabla h-\eta F)$. All \ref{cond:S}--\ref{cond:ST} and \ref{cond:VS+} are satisfied.

    The first-order optimality conditions for \eqref{eq:EG} are:
    \begin{subequations}\label{eq:opt}
        \begin{align}
            \left\langle\nabla h(\omega_t)-\nabla h(x_t) +\eta F(x_t), x-\omega_t\right\rangle&\geq0,\ \forall x \label{eq:opt_first} \\
            \left\langle\nabla h(x_{t+1})-\nabla h(x_t) +\eta F(\omega_t), x-x_{t+1}\right\rangle&\geq0,\ \forall x. \label{eq:opt_second}
        \end{align}
    \end{subequations}
    From \eqref{eq:opt}, we identify
    \begin{subequations}
        \begin{align}
            \omega_t&=\left((\nabla h+\mathcal{N}_\mathcal{X})^{-1}\circ(\nabla h-\eta F)\right)(x_t) \nonumber \\
            &=\left(\nabla h^*\circ(\nabla h-\eta F)\right)(x_t)=T(x_t),\label{eq:EG_w} \\
             \text{and}\quad x_{t+1}&=\nabla h^*(\nabla h(x_t)-\eta F(\omega_t)). \label{eq:EG_x}
        \end{align}
    \end{subequations}
    It remains to show $(S\circ T-S)(x_t)=-\eta F(\omega_t)$, so that \eqref{eq:dual_discrete} reduces to \eqref{eq:EG_x}:
    \begin{subequations}
    \begin{align}
        &S(T(x_t))-S(x_t)=S(\omega_t)-S(x_t) \\
        &=\nabla h(\omega_t) - \eta F(\omega_t) - (\nabla h-\eta F)(x_t)=-\eta F(\omega_t),
    \end{align}
    \end{subequations}
    where the last equality uses \eqref{eq:EG_w}.
\end{proof}

A key observation from \eqref{eq:EG_w} is that the target point $T(x_t)$ coincides exactly with the intermediate point $\omega_t$, revealing that the extragradient correction has a natural interpretation as a target determination. The recent EG+ variant \cite{Diakonikolas2021} generalizes EG by using different step sizes $\eta_1>0$ and $\eta_2>0$ in \eqref{eq:EG_first} and \eqref{eq:EG_second}, respectively. TMD recovers EG+ with $\alpha=\frac{\eta_2}{\eta_1}$.
\begin{corollary}\label{cor:EG+}
    Let $\alpha=\frac{\eta_2}{\eta_1}$, $\beta=0$, $S=\nabla h-\eta_1 F$, and $\Phi=\eta_1 F$. Then, \eqref{eq:dual_discrete} recovers EG+. Moreover, the convergence follows from \Cref{cor:converge}.
\end{corollary}

\begin{remark}[Weak minty variational inequality \cite{Diakonikolas2021}]
    Ref. \cite{Diakonikolas2021} also studies a $L$-Lipschitz $F$ with the following property:
    \begin{align}
        \exists\ x^*\in\operatorname{SOL}(\mathcal{X},F) \text{ s.t. } \langle F(x),x-x^*\rangle\geq -\rho\|F(x)\|^2, \label{eq:WMVI}
    \end{align}
    which slightly relaxes the VS requirement. They prove stability for $\rho\in [0,\frac{1}{4L})$ under $2$-norm. Interestingly, TMD recovers the results since \ref{cond:VS-} is satisfied under this setting. With the selection in \Cref{cor:EG+} and $h$ specialized to Euclidean norm, we have $T(x)-x=-\eta_1 F(x)$. Thus, the condition \eqref{eq:relax} becomes checking the positivity of the following:
    \begin{align}
        \|T(x)-x\|^2-\rho\eta_1\|F(T(x))\|^2. \label{eq:to_check}
    \end{align}
    Rewrite \eqref{eq:to_check} by add-and-subtract techniques and $y=T(x)$:
        \begin{align}
            (1-\theta)\|y-x\|^2+\theta\eta_1^2\|F(x)\|^2-\rho\eta_1\|F(y)\|^2 \label{eq:FY}
        \end{align}
    Bound the last term in \eqref{eq:FY} via Young's inequality, for $t>0$,
    \begin{align}
        \|F(y)\|^2\leq (1+t)L^2\|y-x\|^2+\frac{1+t}{t}\|F(x)\|^2.
    \end{align}
    Therefore, we have that, for $\rho<\frac{1}{4L}$, there exists appropriate $\theta,t,\eta_1$ s.t. \eqref{eq:to_check} is strictly positive for $x\notin\operatorname{SOL}(\mathcal{X},F)$.
\end{remark}

Up to this point, $\Phi$ has appeared intrinsically tied to $F$ (with a scaling factor). The TMD framework, however, admits a broader class of design choices for $\Phi$. In what follows, we explore two such configurations: the first recovers several splitting methods, and the second recovers the Brown-von Neumann-Nash (BNN) dynamics, which establishes a bridge between evolutionary game dynamics and classical optimization algorithms.

\subsubsection{Splitting methods}

In splitting methods, $F$ is decomposed as $A+B$, where $A$ and $B$ possess individually favorable properties. These algorithms commonly assume $A$ and $B$ are maximal monotone\footnote{Refer to \cite[Chapter~12.3 \& Chapter~12.4]{Facchinei2003} for more details.\label{fn:splitting}}, and rely heavily on the resolvent operator\footref{fn:splitting} $J_X=(I+X)^{-1}$, where $I$ is the identity map $I(x)=x$. Two well-known instances are:\\
(i) Douglas-Rachford splitting (DR): \cite[Chapter~12.4.1]{Facchinei2003}
\begin{flalign}
    &x_{t+1}=(J_{\eta A}\circ(2J_{\eta B}-I)+I-J_{\eta B})(x_t)\coloneqq T_{\operatorname{DR}}(x_t).\hspace{-10pt} &\label{eq:DR}
\end{flalign}
(ii) Forward-Backward splitting (FB): \cite[Chapter~12.4.2]{Facchinei2003}, \cite{Bui2021}
\begin{align}
    x_{t+1} =J_{\eta A}\left((I -\eta B)(x_t)\right)\coloneqq T_{\operatorname{FB}}(x_t). \label{eq:FB}
\end{align}
Since the constraints can be absorbed into $A$ or $B$ via indicator functions, we take $\mathcal{X}=\mathbb{R}^n$ in this subsection without loss of generality, and set $\nabla h = I$ for simplicity. A notable distinction from the previous subsections is that $\Phi$ is no longer tied to $F$; instead, we set $\Phi = T_{\operatorname{op}}^{-1}-I$, where $T_{\operatorname{op}}\in\{T_{\operatorname{DR}},T_{\operatorname{FB}}\}$. By this selection, $T_{\operatorname{op}}$ serves as the resolvent of $\Phi$.


\begin{proposition} \label{prop:DR}
    Let $\alpha=1$, $\beta=0$, $S=I$, and $\Phi=T_{\operatorname{DR}}^{-1}-I$. Then, \eqref{eq:dual_discrete} recovers \eqref{eq:DR}. Moreover, the convergence follows from \Cref{cor:converge}.
\end{proposition}

\begin{proof}
    We have $T=(S+\Phi)^{-1}\circ S=T_{\operatorname{DR}}$. Since $T_{\operatorname{DR}}$ is firmly nonexpansive\footref{fn:splitting}, $\Phi$ is (maximal) monotone, satisfying \ref{cond:VS+}. To verify \ref{cond:surrogate}, note that the solutions correspond to the fixed points of $T_{\operatorname{DR}}$, which align with the fixed points of $\Phi$. Substituting the selections into \eqref{eq:dual_discrete} yields \eqref{eq:DR}.
\end{proof}

\begin{proposition} \label{prop:FB} \footnote{Unlike DR, which only requires $A$ and $B$ maximal monotone, FB may require additional conditions, e.g., $B$ Lipschitz and $A+B$ strongly 
monotone.} Let $\alpha=1$, $\beta=0$, $S=I$, and $\Phi=T_{\operatorname{FB}}^{-1}-I$. Then, \eqref{eq:dual_discrete} recovers \eqref{eq:FB}. Moreover, the convergence follows from \Cref{cor:VS-}.
\end{proposition}

\begin{proof}
    The argument follows that of \Cref{prop:DR}, except that $T_{\operatorname{FB}}$ is not firmly nonexpansive, so \ref{cond:VS+} no longer holds and we instead appeal to the relaxed condition \ref{cond:VS-}.

    Let $y=T_{\operatorname{FB}}(x)$, then by definition we have $y+\eta A(y)=x-\eta B(x)$ and $\Phi(y)=x-y=\eta A(y)+\eta B(x)$. Thus, condition \eqref{eq:relax} requires positivity of
    \begin{subequations}\label{eq:check_FB}
    \begin{align}
        &\langle\Phi(y),y-x^*\rangle+\|y-x\|^2 \\
        &=\eta\langle A(y)+B(x), y-x^*\rangle + \|y-x\|^2\\
        &=\eta\langle A(y)+B(y),y-x^*\rangle\nonumber\\
        &\hspace{5em}+\eta\langle B(x)-B(y),y-x^*\rangle+\|y-x\|^2\\
        &\geq \eta \bar{\sigma}\|y-x^*\|^2-\eta \bar{L}\|y-x\| \|y-x^*\|+\|y-x\|^2, \label{eq:apply_AB_conditions}
    \end{align}
    \end{subequations}
    for any fixed points $x^*$ of $T_{\operatorname{FB}}$, where \eqref{eq:apply_AB_conditions} is by $\bar{\sigma}$-strongly monotone of $A+B$, Cauchy-Schwartz, and $\bar{L}$-Lipschitz of $B$. For $\eta<\frac{4\bar{\sigma}}{\bar{L}^2}$, we have \eqref{eq:check_FB} positive unless $x=y=x^*$.
\end{proof}

\subsubsection{Brown-von Neumann-Nash (BNN) dynamics} \cite{Sandholm2010} \\
The BNN dynamics arises from evolutionary game theory and is capable of solving monotone games \cite[Chapter~7.2]{Sandholm2010}. It evolves on the probability simplex $\mathcal{X}=\Delta^n$. Using pre-superscripts to index vector entries, i.e., $x=[{^1x},\ldots,{^nx}]^\top\in\Delta^n$, the BNN dynamics is, for $i=1,\ldots,n$,
\begin{align}
    {^i\dot{x}}(t) = \left[-{^i\hat{F}}(x(t))\right]_+
    - {^ix}(t)\sum_{k=1}^n
    \left[-{^k\hat{F}}(x(t))\right]_+,
    \label{eq:BNN}
\end{align}
where $\hat{F}=F-\bar{F}\mathbf{1}$, with $\bar{F}(x)=x^\top F(x)$ the weighted average of $F$, and $[\cdot]_+$ denotes the entrywise positive part.

In the evolutionary game literature, $[\hat{P}]_+$ is called the \emph{excess payoff} of $P$, and the normalized excess payoff is defined as $[\hat{P}]_+^{\operatorname{nor}}\coloneqq\frac{[\hat{P}(x)]_+}{x}$, where the division is elementwise. The key idea here is to use $[-\hat{F}(x)]_+^{\operatorname{nor}}$ to generate a target point, where the minus sign comes from $F$ being a cost. Starting from the current $x$, map to the dual space via $\nabla h$, shift by $\eta[-\hat{F}(x)]_+^{\operatorname{nor}}$, then map back:
\begin{align}
    T(x) &= \nabla h^*\!\left(\nabla h(x)  + \eta [-\hat{F}(x)]_+^{\operatorname{nor}}\right). \label{eq:BNN_T}
\end{align}
With $h$ chosen as the negative entropy function, which is suitable for the simplex domain, \eqref{eq:BNN_T} admits a closed form, and the TMD dynamics reduces to \eqref{eq:BNN}.

\begin{lemma}\label{lem:BNN}
    Let $\alpha=\frac{1}{\eta}$, $\beta=0$, $S=\nabla h$ with $h$ the negative entropy, and $T$ as in \eqref{eq:BNN_T}. Then, \eqref{eq:dual_update} reduces to \eqref{eq:BNN} and
    \begin{align}
        T(x) = x \oplus \exp\!\left(\eta [-\hat{F}(x)]_+^{\operatorname{nor}}\right) \in \operatorname{int}(\Delta^n), \label{eq:BNN_T_closed_form}
    \end{align}
    where $a\oplus b=\left[\frac{\prescript{1}{}{a}\prescript{1}{}{b}}{\sum_{k} \prescript{k}{}{a} \prescript{k}{}{b}},\ldots,\frac{\prescript{n}{}{a} \prescript{n}{}{b}}{\sum_{k} \prescript{k}{}{a}\prescript{k}{}{b}}\right]^\top$ for $a,b\in\mathbb{R}^n_{>0}$.\footnote{The operator $\oplus$ corresponds to the vector addition in Aitchison geometry.}
\end{lemma}

\begin{proof}
    From \eqref{eq:BNN_T}, we have $(S\circ T-S)(x)=\eta[-\hat{F}(x)]_+^{\operatorname{nor}}$. For $h(x)=\sum_{i=1}^n (\prescript{i}{}{x})\log(\prescript{i}{}{x})$, we get $\nabla h(x)=\log x+\mathbf{1}$, $\nabla h^*(z)=\frac{\exp(z)}{\sum_{k=1}^n \exp(z^i)}$, and $\nabla^2 h^*(\nabla h(x))=\operatorname{diag}(x)-xx^\top$,  where $\log$ and $\exp$ are applied entrywise and $\operatorname{diag}$ is the diagonal matrix. Differentiating $x=\nabla h^*(z)$ gives
    \begin{subequations}
        \begin{align}
            \dot{x}(t)&=\alpha \nabla ^2 h^*(\nabla h (x(t))) (S\circ T-S)(x(t))\\
            &=\left(\operatorname{diag}(x)-xx^T\right)[-\hat{F}(x)]_+^{\operatorname{nor}},
        \end{align}
    \end{subequations}
    which in entrywise form yields \eqref{eq:BNN}. The closed form \eqref{eq:BNN_T_closed_form} follows by direct calculations.
\end{proof}

The preceding subsections have demonstrated TMD's ability 
to recover iterative algorithms via discrete-time reductions. BNN and the following subsections operate instead in continuous-time, where TMD's dynamics \eqref{eq:TMD} apply directly.

\subsubsection{Forward-Backward-Forward dynamics} \cite{Bot2020,Tseng2000}
\begin{subequations}\label{eq:FBF}
    \begin{align}
        y(t)&=P_\mathcal{X}(x(t)-\eta F(x(t))) \\
        \dot{x}(t)&=y(t)+\eta (F(x(t))-F(y(t)))-x(t) \label{eq:FBF_xdot}
    \end{align}
\end{subequations}

\begin{proposition}\label{prop:FBF}
    Let $\alpha=1$, $\beta=0$, $S=I-\eta F$, and $\Phi=\eta F$. Then, \eqref{eq:TMD} reduces to \eqref{eq:FBF}. Moreover,  the convergence follows from \Cref{cor:converge}.
\end{proposition}

\begin{proof}
    With the selections, $T=(I+\mathcal{N}_\mathcal{X})^{-1}\circ(I-\eta F)$. Since $(I+\mathcal{N}_\mathcal{X})^{-1}$ is the projection onto $\mathcal{X}$, we get 
    $y(t)=T(x(t))$. Then, \eqref{eq:FBF_xdot} is equivalent to $\dot{x}(t)=(S\circ T-S)(x(t))$.
\end{proof}

To end this section, we demonstrate how we leverage the TMD framework to improve the discounted mirror descent by correcting a structural limitation in its equilibrium behavior.

\subsubsection{Discounted Mirror Descent (DMD)}
The DMD reads
\begin{subequations} \label{eq:DMD}
\begin{align}
    \dot{z}(t)&= \gamma(-F(x(t))-z(t)) \\
    x(t)&=\nabla h^*(z(t)),
\end{align}
\end{subequations}
with $\gamma>0$. As noted in \cite{Gao2024a}, the equilibrium of \eqref{eq:DMD} does not lie in $\operatorname{SOL}(\mathcal{X},F)$, but in a perturbed solution set. Our approach is to 
precondition $-F$ via the target correction mechanism s.t. the equilibria are recalibrated to match $\operatorname{SOL}(\mathcal{X},F)$.

\begin{proposition}\label{prop:DMD}
    Consider the following two cases:
    \begin{enumerate}
        \item $\alpha=\gamma$, $\beta=0$, $S=\nabla h$, $\Phi=\eta F$
        \item $\alpha=\beta=\gamma$, $S=\nabla h-\eta F$, $\Phi=\eta F$.
    \end{enumerate}
    In either case, \eqref{eq:TMD} reduces to the following calibrated DMD:
    \begin{subequations}\label{eq:cali_DMD}
        \begin{align}
            \dot{z}(t)&=\gamma\!\left(\tilde{F}(x(t))-z(t)\right)\label{eq:DMD_u}\\
            x(t)&=\nabla h^*(z(t)),
        \end{align}
    \end{subequations}
    where $\tilde{F}=S\circ T$. Moreover, the convergence follows from \Cref{cor:converge}.
\end{proposition}
\begin{proof}
    Since $\alpha S+\beta \Phi=\nabla h=z$ in both cases, \eqref{eq:TMD} reduces to \eqref{eq:cali_DMD}. Condition verification follows that in \Cref{prop:PPA} and \Cref{prop:EG} for Case $1$ and Case $2$, respectively.
\end{proof}

\begin{remark}\label{rem:MD2}
    Ref. \cite{Gao2024a} introduces the second-order mirror descent (MD2) to improve DMD, using an equilibrium-independent passivity (EIP) argument. Since \eqref{eq:TMD_u} is output strictly EIP, the same technique applies here and can yield a higher-order TMD: with $\gamma_1,\gamma_2>0$,
    \begin{subequations}\label{eq:TMD_u2}
    \begin{flalign}
        \,\dot{z}(t)&=\alpha(S\circ T-S)(x(t))-\beta\Phi(t)-\gamma_1(x(t)-\xi(t))\hspace{-10pt}&\\
        \dot{\xi}(t)&= \gamma_2(x(t)-\xi(t)),&
    \end{flalign}
    \end{subequations}
    Setting $\alpha=0$ recovers MD2. Moreover, while MD2 requires interior solutions, \eqref{eq:TMD_u2} is free from this restriction.
\end{remark}

\section{Geometric Ensembles}
\label{sec:GE}

In this section, we demonstrate the ability of TMD to run a \emph{geometric ensemble}: multiple algorithms solve the same problem using distinct mirror maps, while their dual updates are driven by a shared ensemble state. Formally, consider $N$ TMD algorithms $\mathcal{A}_i$, each equipped with a distinct mirror map $\nabla (h^i)^*$, primal state $x^i$, and dual state $z^i$, for $i=1,\ldots,N$. Then, the proposed ensemble TMD dynamics is
    \begin{subequations} \label{eq:TMD_ensemble}
    \begin{align}
        \dot{z}^i(t)&=\alpha(S\circ T-S)(x^{\operatorname{en}}(t))\nonumber\\
        &\hspace{6em} -\beta\Phi(x^{\operatorname{en}}(t)),\quad z^i(0)\in\mathbb{R}^n \label{eq:ensemble_TMD_dual}\\
        x^i(t)&= \nabla(h^i)^*(z^i(t)) \label{eq:ensemble_TMD_map}\\
        x^{\operatorname{en}}(t)&=\frac{1}{N} \sum_{k=1}^N x^k(t). \label{eq:ensemble_TMD_avg}
    \end{align}
    \end{subequations}

Beyond its conceptual appeal as an analogue of the wisdom of crowds, the following theorem shows that the ensemble state 
$x^{\operatorname{en}}$ reduces to a single TMD instance with a synthesized mirror map --- hence the name \emph{geometric ensemble}. This is made possible by the decoupling of $\nabla h$ from $S$, $T$, and $\Phi$ in TMD, which allows each $\mathcal{A}_i$ to adopt a distinct mirror map without altering the shared dual update. The synthesized mirror map is given by the infimal convolution, defined as follows.

\begin{definition}[Infimal Convolution {\cite[Chapter~12]{Bauschke2017}}]
    Let $\mathbb{V}$ be a vector space and $f,g:\mathbb{V}\to\bar{\mathbb{R}}=\mathbb{R}\cup\{+\infty\}$. Their infimal convolution, denoted by $f\infconv g:\mathbb{V}\to\bar{\mathbb{R}}$, is
    \begin{align}
        (f\infconv g)(v)=\inf \{f(a)+g(b) : a,b\in\mathbb{V}, a+b=v\}.
    \end{align}
\end{definition}

\begin{lemma}[{\cite[Chapter~12--15]{Bauschke2017}}] \label{lem:inf_conv}
    Let $f,g$ be proper convex functions. Then, $f\infconv g$ is convex and $(f\infconv g)^*=f^*+g^*$.
\end{lemma}

\begin{theorem} \label{thm:ensemble}
    If each $\mathcal{A}^i$ has initial dual state $z^i(0)\in\mathbb{R}^n$, then the ensemble state $x^{\operatorname{en}}$ in \eqref{eq:TMD_ensemble} satisfies
    \begin{subequations}
        \begin{align}
            \dot{z}(t)&=\alpha (S\circ T-S)(x^{\operatorname{en}}(t))-\beta\Phi(x^{\operatorname{en}}(t)) \label{eq:dual_update_shared}\\
            x^{\operatorname{en}}(t)&= \nabla (h^{\operatorname{en}})^*(z(t)), \ \ \text{with}\ \ z(0)=0, \label{eq:geometric_ensemble}
        \end{align}
        where $h^{\operatorname{en}}=\bar{h}^i\infconv\ldots\infconv \bar{h}^N$ with $\bar{h}^i(x)=\frac{1}{N}h^i(Nx)-\langle z^i(0), x\rangle$.
    \end{subequations}
\end{theorem}

\begin{proof}
    Since \eqref{eq:dual_update_shared} is common to all agents, $z^i(t)-z^j(t)$ depends only on initial values and forms a fixed pattern in the dual space. Thus, we can describe the dual update consistently as $\dot{z}=\alpha(S\circ T-S)(x^{\operatorname{en}})-\beta\Phi(x^{\operatorname{en}})$ with $z^i(t)=z(t)+z^i(0)$. Note that, by definition, we get
    \begin{subequations}\label{eq:hbar}
        \begin{align}
            (\bar{h}^i)^*(z)&=\sup_x\! \left\{\langle z, x\rangle-\frac{1}{N}h^i(Nx)+\langle z^i(0), x\rangle\right\}  \\
            &=\sup_u\! \left\{\langle z+z^i(0), \frac{u}{N}\rangle-\frac{1}{N}h^i(u)\right\} \label{eq:change_var}\\
            &=\frac{1}{N} (h^i)^*(z+z^i(0)),
        \end{align}
    \end{subequations}
    where \eqref{eq:change_var} is by $u=Nx.$ Then, \eqref{eq:ensemble_TMD_avg} and \eqref{eq:ensemble_TMD_map} yields
    \begin{subequations}
        \begin{align}
            (h^{\operatorname{en}})^*(z) &=\frac{1}{N}\sum_{k=1}^N (h^k)^*(z+z^k(0))  \\
            &=\sum_{k=1}^N (\bar{h}^k)^*(z)=\left(\bar{h}^1\infconv\ldots\infconv\bar{h}^N\right)^*(z), \label{eq:apply_conv}
        \end{align}
    \end{subequations}
    where \eqref{eq:apply_conv} is by \eqref{eq:hbar} and \Cref{lem:inf_conv}.
\end{proof}

\Cref{thm:ensemble} states that the ensemble behaves as a single TMD with synthesized mirror map $\nabla(h^{\operatorname{en}})^*$, and thus enjoys the convergence results from \Cref{thm:asymp} or \Cref{cor:converge}. The form of the synthesized mirror map indicates that the geometry ensemble not only depends on the distinct mirror maps but also on the initial dual values. This suggests that even with the same mirror map $\nabla h^i = \nabla h$ for all $i$, distinct $z^i(0)$ still yields a nontrivial synthesized mirror map $\nabla(h^{\mathrm{en}})^*$.

\section{Conclusions}
\label{sec:conclusion}
In this paper, we proposed Target Mirror Descent (TMD), a unified framework for solving monotone variational inequalities via a target point correction mechanism in the dual update. We established rigorous convergence guarantees under a set of flexible design conditions, and demonstrated that TMD subsumes several landmark algorithms. Beyond enhancing stability via the feedback correction, TMD decouples the mirror map from the target determination functions, which enables \emph{geometric ensembles}. We showed that the ensemble state 
evolves as a single TMD with a synthesized mirror map determined by the individual mirror maps and 
initial dual states, and thus inherits convergence guarantees. Future work may consider generalizing TMD to non-autonomous systems so that accelerated techniques can be analyzed.

\bibliographystyle{IEEEtran}
\bibliography{reference.bib}

\begin{thebibliography}{10}
\providecommand{\url}[1]{#1}
\csname url@samestyle\endcsname
\providecommand{\newblock}{\relax}
\providecommand{\bibinfo}[2]{#2}
\providecommand{\BIBentrySTDinterwordspacing}{\spaceskip=0pt\relax}
\providecommand{\BIBentryALTinterwordstretchfactor}{4}
\providecommand{\BIBentryALTinterwordspacing}{\spaceskip=\fontdimen2\font plus
\BIBentryALTinterwordstretchfactor\fontdimen3\font minus \fontdimen4\font\relax}
\providecommand{\BIBforeignlanguage}[2]{{%
\expandafter\ifx\csname l@#1\endcsname\relax
\typeout{** WARNING: IEEEtran.bst: No hyphenation pattern has been}%
\typeout{** loaded for the language `#1'. Using the pattern for}%
\typeout{** the default language instead.}%
\else
\language=\csname l@#1\endcsname
\fi
#2}}
\providecommand{\BIBdecl}{\relax}
\BIBdecl

\bibitem{Korpelevich1976}
G.~M. Korpelevich, ``The extragradient method for finding saddle points and other problems,'' \emph{Matecon}, vol.~12, pp. 747--756, 1976.

\bibitem{Mertikopoulos2018a}
P.~Mertikopoulos and M.~Staudigl, ``Stochastic {{Mirror Descent Dynamics}} and {{Their Convergence}} in {{Monotone Variational Inequalities}},'' \emph{Journal of Optimization Theory and Applications}, vol. 179, no.~3, pp. 838--867, 2018.

\bibitem{Eckstein1993}
J.~Eckstein, ``Nonlinear {{Proximal Point Algorithms Using Bregman Functions}}, with {{Applications}} to {{Convex Programming}},'' \emph{Mathematics of Operations Research}, vol.~18, no.~1, pp. 202--226, 1993.

\bibitem{ELFAROUQ2001}
N.~EL~FAROUQ, ``Pseudomonotone {{Variational Inequalities}}: {{Convergence}} of {{Proximal Methods}},'' \emph{Journal of Optimization Theory and Applications}, vol. 109, no.~2, pp. 311--326, 2001.

\bibitem{Nemirovski2004}
A.~Nemirovski, ``Prox-{{Method}} with {{Rate}} of {{Convergence O}}(1/t) for {{Variational Inequalities}} with {{Lipschitz Continuous Monotone Operators}} and {{Smooth Convex-Concave Saddle Point Problems}},'' \emph{SIAM Journal on Optimization}, vol.~15, no.~1, pp. 229--251, 2004.

\bibitem{Dang2015}
C.~D. Dang and G.~Lan, ``On the convergence properties of non-{{Euclidean}} extragradient methods for variational inequalities with generalized monotone operators,'' \emph{Computational Optimization and Applications}, vol.~60, no.~2, pp. 277--310, 2015.

\bibitem{Diakonikolas2021}
J.~Diakonikolas, C.~Daskalakis, and M.~I. Jordan, ``Efficient methods for structured nonconvex-nonconcave min-max optimization,'' in \emph{International Conference on Artificial Intelligence and Statistics}.\hskip 1em plus 0.5em minus 0.4em\relax PMLR, 2021, pp. 2746--2754.

\bibitem{Facchinei2003}
F.~Facchinei and J.-S. Pang, \emph{Finite-Dimensional Variational Inequalities and Complementarity Problems}, ser. Springer Series in Operations Research.\hskip 1em plus 0.5em minus 0.4em\relax Springer, 2003.

\bibitem{Bauschke2017}
H.~H. Bauschke and P.~L. Combettes, \emph{Convex {{Analysis}} and {{Monotone Operator Theory}} in {{Hilbert Spaces}}}, ser. {{CMS Books}} in {{Mathematics}}.\hskip 1em plus 0.5em minus 0.4em\relax Springer International Publishing, 2017.

\bibitem{Bui2021}
M.~N. B{\`u}i and P.~L. Combettes, ``Bregman {{Forward-Backward Operator Splitting}},'' \emph{Set-Valued and Variational Analysis}, vol.~29, no.~3, pp. 583--603, 2021.

\bibitem{Sandholm2010}
W.~H. Sandholm, \emph{Population Games and Evolutionary Dynamics}, ser. Economic Learning and Social Evolution.\hskip 1em plus 0.5em minus 0.4em\relax MIT Press, 2010.

\bibitem{Bot2020}
R.~I. Bo{\c t}, E.~R. Csetnek, and P.~T. Vuong, ``The forward--backward--forward method from continuous and discrete perspective for pseudo-monotone variational inequalities in {{Hilbert}} spaces,'' \emph{European Journal of Operational Research}, vol. 287, no.~1, pp. 49--60, 2020.

\bibitem{Tseng2000}
P.~Tseng, ``A {{Modified Forward-Backward Splitting Method}} for {{Maximal Monotone Mappings}},'' \emph{SIAM Journal on Control and Optimization}, vol.~38, no.~2, pp. 431--446, 2000.

\bibitem{Gao2024a}
B.~Gao and L.~Pavel, ``Second-{{Order Mirror Descent}}: {{Convergence}} in {{Games Beyond Averaging}} and {{Discounting}},'' \emph{IEEE Transactions on Automatic Control}, vol.~69, no.~4, pp. 2143--2157, 2024.

\bibitem{Mokhtari2020}
A.~Mokhtari, A.~Ozdaglar, and S.~Pattathil, ``A {{Unified Analysis}} of {{Extra-gradient}} and {{Optimistic Gradient Methods}} for {{Saddle Point Problems}}: {{Proximal Point Approach}},'' in \emph{Proceedings of the {{Twenty Third International Conference}} on {{Artificial Intelligence}} and {{Statistics}}}.\hskip 1em plus 0.5em minus 0.4em\relax PMLR, 2020, pp. 1497--1507.

\bibitem{Harker1990}
P.~T. Harker and J.-S. Pang, ``Finite-dimensional variational inequality and nonlinear complementarity problems: {{A}} survey of theory, algorithms and applications,'' \emph{Mathematical Programming}, vol.~48, no.~1, pp. 161--220, 1990.

\bibitem{Bauschke2003}
H.~H. Bauschke, J.~M. Borwein, and P.~L. Combettes, ``Bregman {{Monotone Optimization Algorithms}},'' \emph{SIAM Journal on Control and Optimization}, vol.~42, no.~2, pp. 596--636, 2003.

\bibitem{Hines2011}
G.~H. Hines, M.~Arcak, and A.~K. Packard, ``Equilibrium-independent passivity: {{A}} new definition and numerical certification,'' \emph{Automatica}, vol.~47, no.~9, pp. 1949--1956, 2011.

\end{thebibliography}
\end{document}